\documentclass[11pt]{article}
\usepackage{geometry}
\usepackage{amsmath,amssymb,amsthm,amsfonts,}
\usepackage{amsthm,amsmath,amssymb}
\usepackage{mathrsfs}
\usepackage{graphicx}
\usepackage{caption}
\usepackage{float}
\usepackage[scriptsize,nooneline]{subfigure}
\usepackage{indentfirst}
\usepackage{cite}
\usepackage[numbers,sort&compress]{natbib}
\usepackage{color}
\usepackage{epstopdf}
\usepackage{booktabs}
\usepackage{setspace}
\usepackage{arydshln}
\usepackage{enumitem}
\usepackage[utf8]{inputenc}
\usepackage{CJK}
\newtheorem{theorem}{Theorem}[section]
\newtheorem{corollary}{Corollary}[section]
\newtheorem{lemma}{Lemma}[section]

\newtheorem{definition}{Definition}[section]

\newtheorem{proposition}{Proposition}[section]
\newtheorem{remark}{Remark}[section]

\newtheorem{notation}{Notation}[section]
\allowdisplaybreaks[4]
\setenumerate[1]{itemsep=0mm,partopsep=0mm,parsep=\parskip,topsep=1mm}
\setitemize[1]{itemsep=0mm,partopsep=0mm,parsep=\parskip,topsep=1mm}
\setdescription{itemsep=0mm,partopsep=0mm,parsep=\parskip,topsep=1mm}
\begin{document}
\title{The $A_{\alpha}$-spectra of graph operations based on generalized (edge) corona \footnote{This work is supported by the Natural Science Foundation of Xinjiang Province (No. 2021D01C069) and the National Natural Science Foundation of China (No. 12161085).}}
\author{Xiaxia Zhang, Xiaoling Ma\footnote {Corresponding author. Email addresses: mxling2018@163.com}
\\ {\small College of Mathematics and System Sciences Xinjiang University,}\\
{\small Urumqi Xinjiang 830017, P.R.China}\\}
\date{}
\maketitle
\begin{center}
\begin{minipage}{5.5 in}
{\small
{\bf Abstract:} Let $G, H_{i}$ be simple graphs with $n=|V(G)|$, $m=|E(G)|$ and $i=1, 2, \ldots, n(m)$. The generalized corona, denoted $G\tilde{o}\wedge^{n}_{i=1} H_{i}$, is the graph obtained by taking one copy of graphs $G, H_{1},\ldots, H_{n}$ and joining the $i$th vertex of $G$ to every vertex of $H_{i}$ for $1 \leq i \leq n$. The generalized edge corona, denoted by $G[H_i]_1^m$, is the graph obtained by taking one copy of graphs $G, H_{1},\ldots, H_{m}$ and then joining two end-vertices of the $i$th edge of $G$ to every vertex of $H_{i}$ for $1 \leq i \leq m$. For any real $\alpha\in[0,1]$, the matrix $A_{\alpha}(G)=\alpha D(G)+(1-\alpha)A(G)$, where $A(G)$ and $D(G)$ are the adjacency matrix and the degree matrix of a graph $G$, respectively. In this paper, we obtain the $A_{\alpha}$-characteristic polynomial of $G\tilde{o}\wedge^{n}_{i=1} H_{i}$, which extends some known results. Meanwhile, we determine the $A_{\alpha}$-characteristic polynomial of $G[H_i]_1^m$ and get  the $A_{\alpha}$-spectrum of $G[H_i]_1^m$ when $G$ and $H_i$ are regular graphs for $1\le i\le m$. As an application of the above conclusions, we construct infinitely many pairs of non-regular $A_{\alpha}$-cospectral graphs. \\

\noindent {\bf AMS classification:} 05C50, 68R10 \\[1mm]
{\bf Keywords:} Generalized corona; generalized edge corona; $A_{\alpha}$-characteristic polynomial; $A_{\alpha}$-spectrum; $A_{\alpha}$-cospectral}\\
\end{minipage}
\end{center}
\vskip 1cm

\section{Introduction}

In this paper, all graphs considered are simple and undirected.
Let $G$ be a graph of order $n$ with the vertex set $V(G)=\{v_1, v_2, \ldots, v_{n}\}$ and the edge set $E(G)$. The {\it degree} of a vertex $v_i$ in a graph $G$, denoted by $d_i = d_G(v_i)$, is the number of edges incident with $v_i$. Let $D(G) = diag(d_1, d_2, \ldots, d_n)$ be the diagonal matrix of vertex degrees of $G$. The {\it adjacency matrix} of $G$ is an $n\times n$ matrix $A(G)$ whose $(i,j)$-entry is 1 if $v_{i}$ is adjacent to $v_{j}$ and 0, otherwise. The {\it Laplacian matrix} and the {\it signless Laplacian matrix} of $G$ are defined as $L(G)=D(G)-A(G)$ and $Q(G)=D(G)+A(G)$, respectively. In 2017, Nikiforov \cite{ae} presented to study the convex combinations $A_{\alpha}(G)$ of $A(G)$ and $D(G)$ defined by
\begin{equation}\label{A-alpha-matrix}
A_{\alpha}(G)=\alpha D(G)+(1-\alpha)A(G),\quad for \quad 0 \leq \alpha \leq 1.
\end{equation}

It is worth noting that $A_{0}(G)=A(G)$, $A_{1/2}(G)=\frac{1}{2}Q(G)$, $A_{1}(G)=D(G)$ and $A_p(G)-A_q(G)=(p-q)L(G)$. Thus the matrices $A_{\alpha}(G)$ can underpin a unified theory of $A(G)$ and $Q(G)$, and are closely related to the matrix of $L(G)$.

Let $U$ be an $n\times n$ real symmetric matrix. The {\it characteristic polynomial} of the matrix $U$, denoted by $f_{U}(\lambda)=|\lambda I_n-U|$, where $I_{n}$ is the identity matrix of order $n$. The roots of $f_{U}(\lambda)$ are called the {\it $U$-eigenvalues}. The collection of eigenvalues of $U$ together with multiplicities is called the $U$-{\it spectrum},
denoted by $Spec(U)$. Two non-isomorphic graphs are said to be $U$-{\it cospectral} if they have the same $U$-spectrum.

For a graph $G$, according to (\ref{A-alpha-matrix}),
it is observed that the study of $A_{\alpha}$-spectrum not only unifies the related results of $A$-spectrum and $Q$-spectrum but also obtains more comprehensive results and finds more general rules. In addition, unlike the classical graph matrices, the spectral properties of $A_{\alpha}$-matrix vary depending on the variable $\alpha$. Since $A_{\alpha}$-matrix has many properties and characteristics that the classical graph matrices do not have, $A_{\alpha}$-matrix bring many problems worth studying, which enrich and develop the theory of graph spectra. For some related results of $A_{\alpha}$-matrix, we refer the readers to \cite{ae, ai, ao, ap, as, ad, af,ag}.

The information of various spectra in the graph not only reveals the structural properties of the graph, but also has a wide range of applications in computer science and other fields \cite{bs,bf,bt,bn}. Specifically, there are many problems in sensor networks where the tools from the combinatorial optimization and spectral graph theory can help, say in solving partitioning, assignment, routing and scheduling problems \cite{cs,ct,cp}.

Determining the spectra of many graph operations is a basic and very meaningful
work in spectral graph theory. In order to construct a graph whose automorphism group is the wreath product of the automorphism group of their components, Frucht and Harary \cite{br} first introduced the corona of two graphs $G$ and $H$.
Since then a number of results on graph-theoretic properties of corona have appeared. As far as eigenvalues are concerned, the $A$-characteristic polynomial and $Q$-characteristic polynomial of the corona, edge corona, neighborhood corona subdivision-vertex and subdivision-edge neighborhood corona of any two graphs can be expressed by that of two factor graphs \cite{cvet, cd, DH, cx, cw, cv, cve}.

Motivated by the above works, Fiuj Laali, Haj Seyyed Javadi and Dariush Kiani \cite{ba} presented the generalized corona of graphs in Definition \ref{GCG}. Meanwhile, they determined and investigated the characteristic, Laplacian and signless Laplacian polynomial of $G\tilde{o}\wedge^{n}_{i=1} H_{i}$, respectively.

\begin{definition}\label{GCG}
Given simple graphs $G$, $H_1$, $H_2$, \ldots, $H_n$, where $n=|V(G)|$, the
{\it generalized corona}, denoted $G\tilde{o}\wedge^{n}_{i=1} H_{i}$, is the graph obtained by taking one copy of graphs $G$, $H_1$, $H_2$, \ldots, $H_n$ and joining the $i$th vertex of $G$ to every vertex of $H_i$ for $1\le i\le n$.
\end{definition}

For example, let $G$, $H_1$, $H_2$ and $H_3$ be the graphs illustrated in
Fig. \ref{graph1}(a), then the graph $G\tilde{o}\wedge^{n}_{i=1} H_{i}$ can be illustrated in Fig. \ref{graph1}(b).

\begin{figure}[h]
\centering
  \footnotesize
\unitlength 1.75mm 
\linethickness{0.4pt}
\ifx\plotpoint\undefined\newsavebox{\plotpoint}\fi 
\begin{picture}(81.483,28.651)(0,0)
\put(3.525,26.903){\line(0,-1){19.845}}
\put(3.525,27.053){\circle*{2.065}}
\put(3.599,18.058){\circle*{2.07}}
\put(3.45,8.025){\circle*{1.994}}
\put(20.545,25.938){\makebox(0,0)[cc]{$H_{1}$}}
\put(20.545,17.985){\makebox(0,0)[cc]{$H_{2}$}}
\put(20.545,8.1){\makebox(0,0)[cc]{$H_{3}$}}
\put(26.417,26.457){\circle*{1.904}}
\put(34.518,26.384){\circle*{1.892}}
\put(26.491,18.282){\line(1,0){17.318}}
\put(26.453,26.442){\line(1,0){8.063}}
\put(34.266,11.879){\line(-2,-1){8}}
\multiput(42.141,7.817)(-.065655462,.033613445){119}{\line(-1,0){.065655462}}
\put(34.141,11.817){\circle*{1.904}}
\put(42.141,7.879){\circle*{1.875}}
\put(26.266,7.942){\circle*{1.875}}
\put(26.266,7.942){\line(1,0){15.563}}
\put(26.453,18.254){\circle*{1.846}}
\put(43.703,18.192){\circle*{1.803}}
\put(35.078,18.317){\circle*{1.79}}
\put(.074,26.832){\makebox(0,0)[cc]{$v_{1}$}}
\put(.149,17.913){\makebox(0,0)[cc]{$v_{2}$}}
\put(0,7.953){\makebox(0,0)[cc]{$v_{3}$}}
\put(3.345,4.98){\makebox(0,0)[cc]{$G$}}
\put(61.398,26.876){\line(0,-1){19.845}}
\put(61.398,27.024){\circle*{2.065}}
\put(61.472,18.031){\circle*{2.07}}
\put(61.323,7.999){\circle*{1.994}}
\put(75.983,28.024){\line(0,-1){3.875}}
\put(75.858,22.024){\line(0,-1){5.938}}
\multiput(75.858,13.96)(-.03333333,-.08226667){60}{\line(0,-1){.08226667}}
\multiput(73.858,9.024)(.03333333,-.06668333){60}{\line(0,-1){.06668333}}
\put(75.858,5.023){\line(0,1){8.875}}
\put(75.983,27.899){\circle*{1.505}}
\put(75.92,24.274){\circle*{1.425}}
\put(75.795,22.086){\circle*{1.425}}
\put(75.858,16.211){\circle*{1.51}}
\put(75.92,13.837){\circle*{1.398}}
\put(73.92,9.085){\circle*{1.51}}
\put(75.795,5.023){\circle*{1.425}}
\multiput(61.42,27.086)(.5576923,.0336538){26}{\line(1,0){.5576923}}
\multiput(61.545,27.086)(.1722561,-.03353659){82}{\line(1,0){.1722561}}
\multiput(61.545,18.085)(.1255625,.033491071){112}{\line(1,0){.1255625}}
\multiput(61.483,18.085)(.26735185,-.03351852){54}{\line(1,0){.26735185}}
\put(75.858,19.399){\circle*{1.505}}
\multiput(61.42,18.211)(.5111786,.0335){28}{\line(1,0){.5111786}}
\multiput(61.233,8.086)(.085526316,.033631579){171}{\line(1,0){.085526316}}
\multiput(61.42,7.961)(.36582353,.03305882){34}{\line(1,0){.36582353}}
\multiput(61.295,8.023)(.160032967,-.033648352){91}{\line(1,0){.160032967}}
\put(58.42,26.961){\makebox(0,0)[cc]{$v_{1}$}}
\put(58.483,18.024){\makebox(0,0)[cc]{$v_{2}$}}
\put(58.42,8.023){\makebox(0,0)[cc]{$v_{3}$}}
\put(81.42,26.211){\makebox(0,0)[cc]{$H_{1}$}}
\put(81.483,19.399){\makebox(0,0)[cc]{$H_{2}$}}
\put(81.358,9.335){\makebox(0,0)[cc]{$H_{3}$}}
\put(61.393,5.055){\makebox(0,0)[cc]{$G$}}
\put(18.433,0){\makebox(0,0)[cc]{$(a)$}}
\put(68.379,0){\makebox(0,0)[cc]{$(b)$}}
\end{picture}
~\\
~\\
\caption{\footnotesize (a) A connected graph $G$ with vertex set $\{v_1, v_2, v_3\}$ and three graphs $H_1$, $H_2$ and $H_3$. (b) The generalized corona $G\tilde{o}\wedge^{3}_{i=1} H_{i}$.}
\label{graph1}
\end{figure}
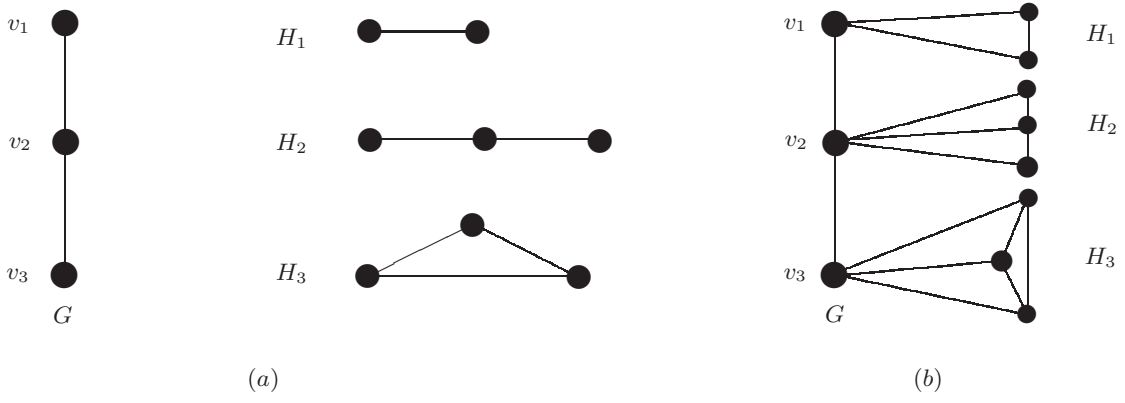

In 2018, Luo and Yan \cite{dl} extended the generalized corona to generalized edge corona and gave the definition of generalized edge corona graphs in Definion \ref{GEC}. In addition, they determined and studied the $A$-characteristic
polynomial, $L$-characteristic polynomial and $Q$-characteristic polynomial of the generalized edge corona graphs, respectively.
\begin{definition}\label{GEC}
Let $G$ be a simple graph with $m$ edges and $H_{1}, H_{2}, \ldots, H_{m}$ be $m$ simple graphs. The generalized edge corona, denoted by $G[H_i]_1^m$, is the graph obtained by taking one copy of graphs $G, H_{1}, H_{2}, \ldots, H_{m}$ and then joining two end-vertices of the $i$th edge of $G$ to every vertex of $H_{i}$ for $1 \leq i \leq m$.
\end{definition}

\begin{figure}[h]
\centering
  \footnotesize
\unitlength 1.75mm 
\linethickness{0.4pt}
\ifx\plotpoint\undefined\newsavebox{\plotpoint}\fi 
\begin{picture}(81.312,27.725)(0,0)
\put(3.079,26.543){\line(0,-1){19.845}}
\put(3.079,26.692){\circle*{2.065}}
\put(3.153,17.698){\circle*{2.07}}
\put(3.004,7.664){\circle*{1.994}}
\put(20.099,25.577){\makebox(0,0)[cc]{$H_{1}$}}
\put(25.971,26.097){\circle*{1.904}}
\put(34.072,26.023){\circle*{1.892}}
\put(26.007,26.081){\line(1,0){8.063}}
\put(20.068,12.437){\makebox(0,0)[cc]{$H_{2}$}}
\put(26.014,12.735){\line(1,0){17.318}}
\put(25.976,12.708){\circle*{1.846}}
\put(43.226,12.644){\circle*{1.803}}
\put(34.601,12.769){\circle*{1.79}}
\put(61.127,26.543){\line(0,-1){19.845}}
\put(61.127,26.692){\circle*{2.065}}
\put(61.201,17.698){\circle*{2.07}}
\put(61.052,7.664){\circle*{1.994}}
\put(58.048,22.149){\makebox(0,0)[cc]{$e_{1}$}}
\put(57.974,12.858){\makebox(0,0)[cc]{$e_{2}$}}
\put(75.961,25.939){\line(0,-1){4.237}}
\put(76.035,17.838){\line(0,-1){8.993}}
\put(75.886,25.939){\circle*{1.863}}
\put(75.961,21.777){\circle*{1.79}}
\put(75.961,17.763){\circle*{1.815}}
\put(75.961,13.155){\circle*{1.904}}
\put(76.035,8.919){\circle*{1.79}}
\multiput(61.17,26.757)(.5048966,-.0333448){29}{\line(1,0){.5048966}}
\multiput(61.021,26.608)(.104263889,-.033548611){144}{\line(1,0){.104263889}}
\multiput(61.244,17.838)(.062042373,.033694915){236}{\line(1,0){.062042373}}
\multiput(61.318,17.763)(.117483871,.033572581){124}{\line(1,0){.117483871}}
\multiput(61.244,17.838)(1.626889,-.033111){9}{\line(1,0){1.626889}}
\multiput(61.244,17.838)(.10375,-.033555556){144}{\line(1,0){.10375}}
\multiput(61.096,17.615)(.0581899225,-.0337054264){258}{\line(1,0){.0581899225}}
\multiput(61.393,7.655)(.050225256,.0337372014){293}{\line(1,0){.050225256}}
\multiput(61.096,7.729)(.096858065,.033567742){155}{\line(1,0){.096858065}}
\multiput(60.873,7.655)(.44594118,.03279412){34}{\line(1,0){.44594118}}
\put(81.312,23.784){\makebox(0,0)[cc]{$H_{1}$}}
\put(81.238,12.783){\makebox(0,0)[cc]{$H_{2}$}}
\put(0,21.777){\makebox(0,0)[cc]{$e_{1}$}}
\put(.074,12.338){\makebox(0,0)[cc]{$e_{2}$}}
\put(2.973,4.905){\makebox(0,0)[cc]{$G$}}
\put(61.17,4.905){\makebox(0,0)[cc]{$G$}}
\put(17.244,0){\makebox(0,0)[cc]{$(a)$}}
\put(68.9,0){\makebox(0,0)[cc]{$(b)$}}
\end{picture}
~\\
~\\
\caption{\footnotesize (a) A connected graph $G$ with the edge set $\{e_1, e_2\}$, $H_1$ and $H_2$. (b) The generalized edge corona $G[H_i]_1^2$.}
\label{graph2}
\end{figure}
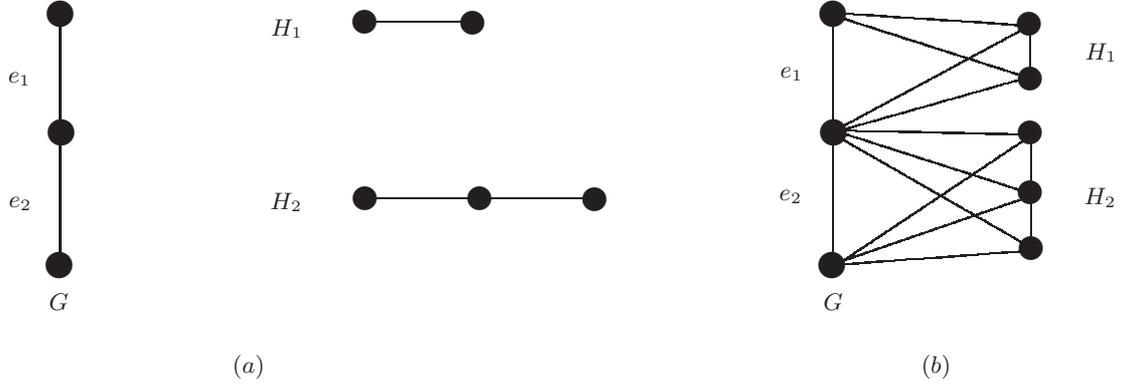

The above graph construction is illustrated using an examples. When $G=P_3$ with the edge set $\{e_1, e_2\}$, $H_1=P_2$ and $H_2=P_3$, the generalized edge corona $G[H_i]_1^2$ is shown as in Fig. \ref{graph2}.

In this paper, we focus on the $A_{\alpha}$-characteristic polynomial of the
generalized corona graphs and generalized edge corona graphs, respectively. In Section 2, we give some preliminaries. In Section 3, for any real $\alpha\in[0,1]$, we compute the characteristic polynomial of $A_{\alpha}(G\tilde{o}\wedge^{n}_{i=1} H_{i})$ and construct infinitely many pairs of non-regular $A_{\alpha}$-cospectral graphs. In Section 4, we obtain the $A_{\alpha}$-characteristic polynomial of $G[H_i]_1^m$ and get  the $A_{\alpha}$-spectrum of $G[H_i]_1^m$ when $G$ and $H_i$ are regular graphs for $1\le i\le m$. As an application, we also consider the construction of non-regular $A_{\alpha}$-cospectral graphs.

\section{Preliminaries}
In order to prove our main results, we state several known conclusions, all of which are used in the following Sections. Throughout the paper for any integers $k$, $n_1$ and $n_2$, $I_k$ denotes the identity matrix of size $k$, $\mathbf{1}_k$ denotes the column vector of size $k$ whose all entries are $1$ and $J_{n_1\times n_2}$ denotes the $n_1\times n_2$ matrix whose all entries are $1$.

\begin{lemma}\label{Schur}[Schur complement\cite{da}] Let $A$ be an $n \times n$ matrix partitioned as
$$
\left(\begin{array}{ll}
A_{11} & A_{12} \\
A_{21} & A_{22}
\end{array}\right),
$$
where $A_{11}$ and $A_{22}$ are square matrices. If $A_{11}$ and $A_{22}$ are invertible, then
$$
\begin{aligned}
\operatorname{det}\left(\begin{array}{ll}
A_{11} & A_{12} \\
A_{21} & A_{22}
\end{array}\right) & =\operatorname{det}\left(A_{22}\right) \operatorname{det}\left(A_{11}-A_{12} A_{22}{ }^{-1} A_{21}\right) \\
& =\operatorname{det}\left(A_{11}\right) \operatorname{det}\left(A_{22}-A_{21} A_{11}{ }^{-1} A_{12}\right).
\end{aligned}
$$
\end{lemma}

\begin{proposition}\cite{ba}\label{Wmatrix}
Let $G$ be a bipartite graph of order $n$ with $V(G)=(P, Q)$ such that $|P|=p$ and $|Q|=n-p$. Also assume that $W$ is a matrix of order $p \times(n-p)$ such that
$$
A(G)=\left(\begin{array}{cc}
0 & W \\
W^T & 0
\end{array}\right).
$$
For
$$
f_A(\lambda)=\operatorname{det}\left(\begin{array}{cc}
\lambda I_p & -W \\
-W^T & \lambda I_{n-p}
\end{array}\right),
$$
with two methods of Schur complement we have
$$
\begin{aligned}
f_A(\lambda) & =\lambda^{n-2 p} \operatorname{det}\left(\lambda^2 I_p-W W^T\right)=\lambda^{n-2 p} l(\lambda) \\
& =\lambda^{2 p-n} \operatorname{det}\left(\lambda^2 I_{n-p}-W^T W\right)=\lambda^{2 p-n} h(\lambda).
\end{aligned}
$$
It is clear that if $n=2p$, then $f_A(\lambda)=l(\lambda)=h(\lambda)$.
\end{proposition}

For any real $\alpha\in[0,1]$, let $A_{\alpha}(G)=\alpha D(G)+(1-\alpha)A(G)$, where $A(G)$ and $D(G)$ are the adjacency matrix and the degree matrix of a graph $G$ with $n$ vertices, respectively. Based on the definition of $M$-coronal of a matrix $M$\cite{cw}, Abdus Sahir and Abu Nayeem \cite{dq} introduced the $A_{\alpha}$-{\it coronal} $\chi_{A_{\alpha}(G)}(\lambda)$ of $G$, which is defined to be the sum of all entries of the matrix $\big((\lambda-\alpha) I_n-A_{\alpha}(G)\big)^{-1}$, that is $\chi_{A_{\alpha}(G)}(\lambda)=\mathbf{1}_n^T\big((\lambda-\alpha) I-A_{\alpha}(G)\big)^{-1} \mathbf{1}_n$.

Tahir and Zhang \cite{dc} stated the following $\chi_{A_{\alpha}(G)}(\lambda)$ when $G$ is a regular graph or a complete bipartite graph. Next, we also give a short proof of $\chi_{A_{\alpha}(G)}(\lambda)$ for a regular graph.

\begin{lemma}\label{r-regular}
Let $G$ be an $r$-regular graph of order $n$. For any real $\alpha\in[0,1]$, then
$$\chi_{A_{\alpha}(G)}(\lambda)=\frac{n}{\lambda-r-\alpha}.$$
\end{lemma}

\begin{proof}
By the definition of $A_{\alpha}$-coronal, we have $\chi_{A_{\alpha}(G)}(\lambda)=\mathbf{1}_n^T\big((\lambda-\alpha) I-A_{\alpha}(G)\big)^{-1} \mathbf{1}_n.$ Since $G$ is an $r$-regular graph of order $n$, it is easy to see that $D(G)=rI_n$. Thus
$$\begin{aligned}
\big((\lambda-\alpha)I_n-A_{\alpha}(G)\big) \mathbf{1}_n
&=\big((\lambda-\alpha)I_n-\alpha D(G)-(1-\alpha)A(G)\big) \mathbf{1}_n\\
&=(\lambda-\alpha)\mathbf{1}_n-\alpha r \mathbf{1}_n-(1-\alpha)r \mathbf{1}_n \\
&=\big(\lambda-\alpha-\alpha r-(1-\alpha)r\big) \mathbf{1}_n \\
&=(\lambda-\alpha-r)\mathbf{1}_n,
\end{aligned}
$$
which implies that
$$\begin{aligned}
\frac{1}{\lambda-r-\alpha}\mathbf{1}_n&=\big((\lambda-\alpha)I-\alpha D(G)-(1-\alpha)A(G)\big)^{-1} \mathbf{1}_n\\
\quad \frac{\mathbf{1}_n^T \cdot \mathbf{1}_n}{\lambda-r-\alpha}&=\mathbf{1}_n^T\big((\lambda-\alpha)I-\alpha D(G)-(1-\alpha)A(G)\big)^{-1} \mathbf{1}_n\\
\chi_{A_{\alpha}(G)}(\lambda)&=\frac{n}{\lambda-r-\alpha},
\end{aligned}$$
as required.
\end{proof}

\begin{lemma}\cite{dc}\label{chi}
For the complete bipartite graph $K_{p,q}$ with two parts of order $p, q$, respectively, the $A_{\alpha}$-coronal of $K_{p.q}$ is given by
$$\chi_{A_{\alpha}(K_{p,q})}(\lambda)=\frac{(p+q)\lambda-\alpha (p+q)^{2}+2pq}{\lambda^{2}
-\alpha (p+q)\lambda+(2\alpha-1)pq}.$$
\end{lemma}


\section{$A_{\alpha}$-characteristic polynomial of $G\tilde{o}\wedge^{n}_{i=1} H_{i}$}
In this section, we will first concentrate on the $A_{\alpha}$-characteristic polynomial of the generalized corona graph and then consider some applications of related results.


\begin{notation}
Let $G$ be a graph of order $n$. For graphs $H_1, H_2, \ldots, H_n$ of orders $t_1, t_2, \ldots, t_n$, respectively, we define the $n$-variable polynomial
$$\begin{aligned}
& g_{A_{\alpha}}\big(\chi_{A_{\alpha}(H_1)}(\lambda), \chi_{A_{\alpha}(H_2)}(\lambda), \ldots, \chi_{A_{\alpha}(H_n)}(\lambda); G\big) \\
& =\operatorname{det}\left(\left[\begin{array}{ccc}
\lambda-\alpha t_1-(1-\alpha)^{2}\chi_{A_{\alpha}(H_1)}(\lambda) & 0 & 0 \\
0 & \ddots & 0 \\
0 & 0 & \lambda-\alpha t_n-(1-\alpha)^{2}\chi_{A_{\alpha}(H_n)}(\lambda)
\end{array}\right]-A_{\alpha}(G)\right) . \\
\end{aligned}$$
\end{notation}

\begin{proposition}\label{same-order}
Let $G$ be a graph of order $n$ and $H_1, H_2, \ldots, H_n$ be $n$ graphs with same order $q$. If
$\chi_{A_{\alpha}\left(H_1\right)}(\lambda)=\chi_{A_{\alpha}\left(H_2\right)}(\lambda)=\cdots=\chi_{A_{\alpha}\left(H_n\right)}(\lambda)
=\chi(\lambda)$,
then
$$
\begin{gathered}
g_{A_{\alpha}}\big(\chi_{A_{\alpha}(H_1)}(\lambda), \chi_{A_{\alpha}(H_2)}(\lambda), \ldots, \chi_{A_{\alpha}(H_n)}(\lambda); G\big)
=f_{A_{\alpha}(G)}\big(\lambda-\alpha q-(1-\alpha)^{2}\chi(\lambda)\big).
\end{gathered}
$$
\end{proposition}

\begin{theorem}\label{main}
Let $G$ be a graph of order $n$ with the vertex set $V(G)=\{v_1, v_2, \ldots, v_n\}$ and $H_1, H_2, \ldots, H_n$ be $n$ graphs (not necessarily non-isomorphic) of orders $t_1, t_2, \ldots, t_n$, respectively. Then, for any real $\alpha\in[0,1]$, the $A_{\alpha}$-characteristic polynomial of $G\tilde{o}\wedge^{n}_{i=1} H_{i}$ as follows:
$$
f_{A_{\alpha}(G\tilde{o}\wedge^{n}_{i=1} H_{i})}(\lambda)=\prod_{i=1}^n f_{A_{\alpha}(H_i)}(\lambda-\alpha)\cdot
g_{A_{\alpha}}(\chi_{A_{\alpha}(H_1)}(\lambda), \chi_{A_{\alpha}(H_2)}(\lambda), \ldots, \chi_{A_{\alpha}(H_n)}(\lambda); G\big).
$$
\end{theorem}

\begin{proof}
By the Definition \ref{GCG}, we have the adjacency matrix of $G\tilde{o}\wedge^{n}_{i=1} H_{i}$ is given as follows:
\begin{equation}\label{AD-Matrix}
A(G\tilde{o}\wedge^{n}_{i=1} H_{i})=\begin{pmatrix}
A(G) & C \\
C^{T} & K
\end{pmatrix},
\end{equation}
where $C$ is a matrix of order $n\times(t_{1}+t_{2}+\cdots+t_{n})$, $K$ is a square matrix of order $t_{1}+t_{2}+\cdots+t_{n}$ and
\begin{footnotesize}
\[K=\begin{array}{c}
\end{array} \bordermatrix{
       &V(H_{1})&V(H_{2})&\cdots&V(H_{n-1})&V(H_{n})\cr
 &A(H_{1}) & 0 & 0 & \cdots & 0 \cr
 &0 & A(H_{2}) & 0 & \cdots & 0 \cr
 &\vdots & \vdots & \ddots & \vdots & \vdots \cr
&0 & 0 & \cdots & A(H_{n-1}) & 0 \cr
 &0 & 0 & \cdots & 0 & A(H_{n}) \cr} ,C=\begin{array}{c}
v_{1} \\
v_{2} \\
\vdots \\
v_{n-1} \\
v_{n}
\end{array} \bordermatrix{
 &V(H_{1})&V(H_{2})&\cdots&V(H_{n-1})&V(H_{n})\cr
 &\mathbf{1}_{t_{1}}^{T} & 0 & 0 & \cdots & 0 \cr
 &0 & \mathbf{1}_{t_{2}}^{T} & 0 & \cdots & 0 \cr
 &\vdots & \vdots & \ddots & \vdots & \vdots \cr
 &0 & 0 & \cdots & \mathbf{1}_{t_{n-1}}^{T} & 0 \cr
 &0 & 0 & \cdots & 0 & \mathbf{1}_{t_{n}}^{T} \cr}. \]
\end{footnotesize}

Let $U=diag(t_{1}, t_2, \ldots, t_{n})$. Similarly, we have the degree matrix of $G\tilde{o}\wedge^{n}_{i=1} H_{i}$ is obtained as follows:
\begin{equation}\label{De-matrix}
D(G\tilde{o}\wedge^{n}_{i=1} H_{i})=\begin{array}{c}
V(G)\\
V(H_1) \\
V(H_2) \\
\vdots \\
V(H_n)
\end{array} \bordermatrix{
  &V(G) &V(H_{1})&V(H_{2})&\cdots&V(H_{n})\cr
  &D(G)+U&0&0&\cdots&0\cr
  &0          &D(H_{1})+I_{t_1} & 0 & \cdots & 0 \cr
  &0          &0 &D(H_{2})+I_{t_2}  & \cdots& 0  \cr
  &\vdots     &\vdots & \vdots & \ddots & \vdots  \cr
  &0          &0 & 0 & \cdots &D(H_{n})+I_{t_n}  \cr
              }.
\end{equation}
Using (\ref{A-alpha-matrix}), we know
%
\begin{equation}\label{GCA-matrix}
A_{\alpha}(G\tilde{o}\wedge^{n}_{i=1} H_{i})=\alpha D(G\tilde{o}\wedge^{n}_{i=1} H_{i})+(1-\alpha)A(G\tilde{o}\wedge^{n}_{i=1} H_{i}).
\end{equation}

Substituting (\ref{AD-Matrix}) and (\ref{De-matrix}) into (\ref{GCA-matrix}), we obtain the $A_{\alpha}$-matrix of $G\tilde{o}\wedge^{n}_{i=1} H_{i}$ below
$$A_{\alpha}(G\tilde{o}\wedge^{n}_{i=1} H_{i})=\begin{pmatrix}
\alpha U+A_{\alpha}(G) & (1-\alpha)C \\
(1-\alpha)C^{T} & W
\end{pmatrix},$$
where
$$W=
\left(
\begin{array}{ccccccccccccc}
& \alpha I_{t_1}+A_{\alpha}(H_1) & 0 &\cdots & 0 \\
& 0 & \alpha I_{t_2}+A_{\alpha}(H_2)& \cdots & 0 \\
& 0 & 0 & \cdots & 0 \\
&0 & 0 & \cdots &\alpha I_{t_n}+A_{\alpha}(H_n)
\end{array}
\right).
$$

Thus, it follows that
\begin{flushleft}
$f_{A_{\alpha}(G\tilde{o}\wedge^{n}_{i=1} H_{i})}(\lambda)=\operatorname{det}\left(\begin{array}{cc}
\lambda I_{n}-\alpha U-A_{\alpha}(G)&-(1-\alpha)C \\
-(1-\alpha)C^T &\lambda I_{(t_{1}+t_{2}+\cdots+t_{n})}-W
\end{array}\right)$

$
=\operatorname{det}\left(\begin{array}{cc}
\lambda I_{n}-\alpha U-A_{\alpha}(G)  & -(1-\alpha)C \\
-(1-\alpha)C^{T} & {\left[\begin{array}{ccc}
\lambda I_{t_{1}}-\alpha I_{t_{1}}-A_{\alpha}(H_{1}) & 0 & 0 \\
0 & \ddots & 0 \\
0 & 0 & \lambda I_{t_{n}}-\alpha I_{t_{n}}-A_{\alpha}(H_{n}))
\end{array}\right]}
\end{array}\right).
$
\end{flushleft}

Applying Lemma \ref{Schur}, we obtain that
\begin{align*}
&f_{A_{\alpha}(G\tilde{o}\wedge^{n}_{i=1}H_{i})}(\lambda)\\
=&\prod_{i=1}^n f_{A_{\alpha}(H_i)}(\lambda-\alpha) \\
&\cdot\operatorname{det}\left(\lambda I_{n}-\alpha U-A_{\alpha}(G)-(1-\alpha)^{2}C\left[\begin{array}{ccc}
\left(\lambda-\alpha)I_{t_{1}}-A_{\alpha}(H_1\right) & 0 & 0 \\
0 & \ddots & 0 \\
0 & 0 & (\lambda-\alpha)I_{t_{n}}-A_{\alpha}(H_n)
\end{array}\right]^{-1} C^T\right) \\
=&\prod_{i=1}^n f_{A_{\alpha}(H_i)}(\lambda-\alpha) \\
&\cdot\operatorname{det}\left(\lambda I_{n}-\alpha U-A_{\alpha}(G)-(1-\alpha)^{2}C
\begin{scriptsize}
\left[\begin{array}{ccc}
\big(\left(\lambda-\alpha)I_{t_{1}}-A_{\alpha}(H_1\right)\big)^{-1} & 0 & 0 \\
0 & \ddots & 0 \\
0 & 0 & \big((\lambda-\alpha)I_{t_{n}}-A_{\alpha}(H_n)\big)^{-1}
\end{array}\right]\end{scriptsize} C^T\right)\\
=&\prod_{i=1}^n f_{A_{\alpha}(H_i)}(\lambda-\alpha) \\
&\cdot\operatorname{det}\left(\lambda I_{n}-\alpha U-A_{\alpha}(G)
-(1-\alpha)^{2}\begin{scriptsize}\left[\begin{array}{ccc}
\mathbf{1}_{t_{1}}^{T}\big((\lambda-\alpha)I_{t_{1}}-A_{\alpha}(H_1)\big)^{-1}\mathbf{1}_{t_{1}} & 0 & 0 \\
0 & \ddots & 0 \\
0 & 0 & \mathbf{1}_{t_{n}}^{T}\big((\lambda-\alpha)I_{t_{n}}-A_{\alpha}(H_n)\big)^{-1}\mathbf{1}_{t_{n}}
\end{array}\right]\end{scriptsize}\right)\\
=&\prod_{i=1}^n f_{A_{\alpha}(H_i)}(\lambda-\alpha) \\
&\cdot\operatorname{det}\left(\lambda I_{n}-\alpha U-A_{\alpha}(G)
-(1-\alpha)^{2}\left[\begin{array}{ccc}
\chi_{A_{\alpha}(H_{1})}(\lambda) & 0 & 0 \\
0 & \ddots & 0 \\
0 & 0 & \chi_{A_{\alpha}(H_{n})}(\lambda)
\end{array}\right]\right)\\
=&\prod_{i=1}^n f_{A_{\alpha}(H_i)}(\lambda-\alpha) \\
&\cdot\operatorname{det}\left(\left[\begin{array}{ccc}
\lambda-\alpha t_{1}-(1-\alpha)^{2}\chi_{A_{\alpha}(H_{1})}(\lambda) & 0 & 0 \\
0 & \ddots & 0 \\
0 & 0 & \lambda-\alpha t_{n}-(1-\alpha)^{2}\chi_{A_{\alpha}(H_{n})}(\lambda)
\end{array}\right]-A_{\alpha}(G)\right).
\end{align*}

By Notation 3.1, for any real $\alpha\in[0,1]$, it follows that
\begin{equation*}
f_{A_{\alpha}(G\tilde{o}\wedge_{i=1}^n
H_{i})}(\lambda)=\prod_{i=1}^n f_{A_{\alpha}(H_i)}(\lambda-\alpha)\cdot
g_{A_{\alpha}}\big(\chi_{A_{\alpha}(H_1)}(\lambda),\ldots,\chi_{A_{\alpha}(H_n)}(\lambda);G\big).
\end{equation*}
This completes the proof.
\end{proof}

\begin{remark}
The adjacency and the signless Laplacian characteristic polynomial of $A(G\tilde{o}\wedge^{n}_{i=1}H_{i})$ can be easily obtained by Theorem 3.1 for $\alpha=0$ and $\alpha=\frac{1}{2}$, respectively.
\end{remark}

Frucht and Harary \cite{br} first introduced the corona of $G$ and $H$ with the goal of constructing a graph whose automorphism group is the wreath product of the automorphism group of their components. The {\it corona} of $G$ and $H$, denoted $G\circ H$, is the graph obtained by taking one copy of $G$ and $|V(G)|$
copies of $H$, and joining the $i$th vertex of $G$ to every vertex in the $i$th copy of $H$. Therefore, combining Definition \ref{GCG} and Theorem \ref{main}, we have the following corollary.

\begin{corollary}
Let $G$ and $H$ be two graphs of orders $n$ and $n'$, respectively. Then, for any real $\alpha\in[0,1]$, we have
$$
f_{A_{\alpha}(G \circ H)}(\lambda)=\left(f_{A_{\alpha}(H)}(\lambda-\alpha)\right)^n \cdot f_{A_{\alpha}(G)}\left(\lambda-\alpha n'-(1-\alpha)^{2}\chi_{A_{\alpha}(H)}(\lambda)\right) .
$$
\end{corollary}

Specially, if $H_1, H_2, \ldots, H_n$ are $n$ graphs with same order $q$ and same $A_{\alpha}$-coronal, i.e.
$\chi_{A_{\alpha}\left(H_1\right)}(\lambda)=\chi_{A_{\alpha}\left(H_2\right)}(\lambda)=\cdots=\chi_{A_{\alpha}\left(H_n\right)}(\lambda)
=\chi(\lambda)$, then by Proposition \ref{same-order} and Theorem \ref{main}, we know the corollary below.
\begin{corollary}
Let $G$ be a graph of order $n$ and $H_1, H_2, \ldots, H_n$ be $n$ graphs each of them has order $q$ such that $\chi_{A_{\alpha}(H_{1})}(\lambda)=\cdots=\chi_{A_{\alpha}(H_{n})}(\lambda)
=\chi(\lambda)$. For any real $\alpha\in[0,1]$, then
$$
f_{A_{\alpha}(G\tilde{o}\wedge^{n}_{i=1}H_{i})}(\lambda)
=\big(\prod_{i=1}^n f_{A_{\alpha}(H_{i})}(\lambda-\alpha)\big) \cdot f_{A_{\alpha}(G)}\left(\lambda-\alpha q-(1-\alpha)^{2}\chi(\lambda)\right) .
$$

Particularly, if $H_1, H_2, \ldots, H_n$ be $r$-regular graphs each of them has order $q$, then by Lemma \ref{r-regular}, we have
$$
f_{A_{\alpha}(G\tilde{o}\wedge^{n}_{i=1}H_{i})}(\lambda)
=\big(\prod_{i=1}^n f_{A_{\alpha}(H_{i})}(\lambda-\alpha)\big) \cdot f_{A_{\alpha}(G)}\left(\lambda-\alpha q-(1-\alpha)^{2}\frac{n'}{\lambda-r-\alpha}\right).
$$
\end{corollary}

%
As an application of above results, we construct infinitely many pairs of non-regular $A_{\alpha}$-cospectral graphs as follows.

\begin{corollary}
Let $G_{1}$ and $G_{2}$ be two $A_{\alpha}$-cospectral graphs of order $n$ and $H_1, H_2, \ldots, H_n$ are $n$ graphs with same order $q$ and same $A_{\alpha}$-coronal. Then, for any real $\alpha\in[0,1]$, $G_{1}\tilde{o}\wedge^{n}_{i=1}H_{i}$ and $G_{2}\tilde{o}\wedge^{n}_{i=1}H_{i}$ are $A_{\alpha}$-cospectral.
\end{corollary}

\begin{corollary}
Let $G$ be a graphs of order $n$ and $H_1, H_2, \ldots, H_{2n}$ be a family of $A_{\alpha}$-cospectral graphs such that $\chi_{A_{\alpha}(H_{1})}(\lambda)=\cdots=\chi_{A_{\alpha}(H_{2n})}(\lambda)
=\chi(\lambda)$. Then, for any real $\alpha\in[0,1]$, $G\tilde{o}\wedge^{n}_{i=1}H_{i}$ and $G\tilde{o}\wedge^{2n}_{i=n+1}H_{i}$ are $A_{\alpha}$-cospectral.
\end{corollary}

\begin{theorem}\label{pq bipartite}
Let $G$ be an $(r_1,r_2)$-semiregular bipartite graph and let $V(G)=(P, Q)$ with $|P|=p$ and $|Q|=n-p$. Suppose $H_1\cong H_2\cong \cdots \cong H_p \cong Z_1$ with $|V(Z_1)|=n_{1}$ and $H_{p+1} \cong \cdots \cong H_n \cong Z_2$ with $|V(Z_2)|=n_{2}$. Then, for any real $\alpha\in[0,1]$,
\begin{align*}
&f_{A_{\alpha}(G\tilde{o}\wedge^{n}_{i=1}H_{i})}(\lambda)\\
&=\big(f_{A_{\alpha}(Z_{1})}(\lambda-\alpha)\big)^p \cdot \big(f_{A_{\alpha}(Z_2)}(\lambda-\alpha)\big)^{n-p}\cdot
\big(\lambda-\alpha n_{2}-\alpha r_{2}-(1-\alpha)^{2}\chi_{A_{\alpha}(Z_{2})}(\lambda)\big)^{n-2p}
\\
&(1-\alpha)^{2p}\cdot l\left(\frac{\sqrt{\big(\lambda-\alpha n_{1}- \alpha r_{1}-(1-\alpha)^{2}\chi_{A_{\alpha}(Z_{1})}(\lambda)\big)\cdot\big(\lambda-\alpha n_{2}- \alpha r_{2}-(1-\alpha)^{2}\chi_{A_{\alpha}(Z_{2})}(\lambda)\big)}}{1-\alpha}\right),
\end{align*}
or
\begin{align*}
&f_{A_{\alpha}(G\tilde{o}\wedge^{n}_{i=1}H_{i})}(\lambda)\\
&=\big(f_{A_{\alpha}(Z_{1})}(\lambda-\alpha)\big)^p \cdot \big(f_{A_{\alpha}(Z_2)}(\lambda-\alpha)\big)^{n-p}\cdot
\big(\lambda-\alpha n_{1}-\alpha r_{1}-(1-\alpha)^{2}\chi_{A_{\alpha}(Z_{1})}(\lambda)\big)^{2p-n}\\
& (1-\alpha)^{2(n-p)}\cdot h\left(\frac{\sqrt{\big(\lambda-\alpha n_{1}-\alpha r_{1}-(1-\alpha)^{2}\chi_{A_{\alpha}(Z_{1})}(\lambda)\big)\cdot\big(\lambda-\alpha n_{2}-\alpha r_{2}-(1-\alpha)^{2}\chi_{A_{\alpha}(Z_{2})}(\lambda)\big)}}{1-\alpha}\right).
\end{align*}
\end{theorem}

\begin{proof}
By the same argument given in the proof of Theorem \ref{main}, we have
$$
\begin{aligned}
& f_{A_{\alpha}(G\tilde{o}\wedge^{n}_{i=1}H_{i})}(\lambda)\\
=&\big(f_{A_{\alpha}(Z_{1})}(\lambda-\alpha)\big)^p \cdot \big(f_{A_{\alpha}(Z_2)}(\lambda-\alpha)\big)^{n-p} \\
& \cdot \operatorname{det}\left(\left[\begin{array}{cc}
(\lambda-\alpha n_{1}-(1-\alpha)^{2}\chi_{A_{\alpha}(Z_1)}(\lambda)) I_{p} & 0 \\
0 & (\lambda-\alpha n_{2}-(1-\alpha)^{2}\chi_{A_{\alpha}(Z_2)}(\lambda)) I_{n-p}
\end{array}\right]-A_{\alpha}(G)\right)\\
=&\big(f_{A_{\alpha}(Z_{1})}(\lambda-\alpha)\big)^p \cdot \big(f_{A_{\alpha}(Z_2)}(\lambda-\alpha)\big)^{n-p} \\
& \cdot \operatorname{det}\left(\begin{footnotesize}\left[\begin{array}{cc}
\big(\lambda-\alpha n_{1}-(1-\alpha)^{2}\chi_{A_{\alpha}(Z_1)}(\lambda)\big)I_{p} & 0 \\
0 & \big(\lambda-\alpha n_{2}-(1-\alpha)^{2}\chi_{A_{\alpha}(Z_2)}(\lambda)\big)I_{n-p}
\end{array}\right]\end{footnotesize}-\alpha D(G)-(1-\alpha)A(G)\right).\\
\end{aligned}
$$

Since $G$ is an $(r_1,r_2)$-semiregular bipartite graph, there is a matrix $Y$ of order $p\times(n-p)$ such that
$$
\begin{aligned}
& f_{A_{\alpha}(G\tilde{o}\wedge^{n}_{i=1}H_{i})}(\lambda)\\
=&\big(f_{A_{\alpha}(Z_{1})}(\lambda-\alpha)\big)^p \cdot \big(f_{A_{\alpha}(Z_2)}(\lambda-\alpha)\big)^{n-p} \\
& \cdot \operatorname{det}\left(\left[\begin{array}{cc}
\big(\lambda-\alpha n_{1}-
\alpha r_{1}-(1-\alpha)^{2}\chi_{A_{\alpha}(Z_1)}(\lambda)\big)I_{p} &-(1-\alpha)Y \\
-(1-\alpha)Y^{T} & \big(\lambda-\alpha n_{2}-
\alpha r_{2}-(1-\alpha)^{2}\chi_{A_{\alpha}(Z_2)}(\lambda)\big)I_{n-p}
\end{array}\right]\right).\\
\end{aligned}
$$
Therefore, by Lemma \ref{Schur} and Proposition \ref{Wmatrix}, we obtain that
\begin{align*}
&f_{A_{\alpha}(G\tilde{o}\wedge^{n}_{i=1}H_{i})}(\lambda)\\
&=\big(f_{A_{\alpha}(Z_{1})}(\lambda-\alpha)\big)^p \cdot \big(f_{A_{\alpha}(Z_2)}(\lambda-\alpha)\big)^{n-p}\cdot
\big(\lambda-\alpha n_{2}-\alpha r_{2}-(1-\alpha)^{2}\chi_{A_{\alpha}(Z_{2})}(\lambda)\big)^{n-2p}
\\
&(1-\alpha)^{2p}\cdot l\left(\frac{\sqrt{\big(\lambda-\alpha n_{1}- \alpha r_{1}-(1-\alpha)^{2}\chi_{A_{\alpha}(Z_{1})}(\lambda)\big)\cdot\big(\lambda-\alpha n_{2}- \alpha r_{2}-(1-\alpha)^{2}\chi_{A_{\alpha}(Z_{2})}(\lambda)\big)}}{1-\alpha}\right),
\end{align*}
or
\begin{align*}
&f_{A_{\alpha}(G\tilde{o}\wedge^{n}_{i=1}H_{i})}(\lambda)\\
&=\big(f_{A_{\alpha}(Z_{1})}(\lambda-\alpha)\big)^p \cdot \big(f_{A_{\alpha}(Z_2)}(\lambda-\alpha)\big)^{n-p}\cdot
\big(\lambda-\alpha n_{1}-\alpha r_{1}-(1-\alpha)^{2}\chi_{A_{\alpha}(Z_{1})}(\lambda)\big)^{2p-n}\\
& (1-\alpha)^{2(n-p)}\cdot h\left(\frac{\sqrt{\big(\lambda-\alpha n_{1}-\alpha r_{1}-(1-\alpha)^{2}\chi_{A_{\alpha}(Z_{1})}(\lambda)\big)\cdot\big(\lambda-\alpha n_{2}-\alpha r_{2}-(1-\alpha)^{2}\chi_{A_{\alpha}(Z_{2})}(\lambda)\big)}}{1-\alpha}\right).
\end{align*}
\end{proof}

\begin{corollary}\label{KS}
Let $G$ be an $(r_1,r_2)$-semiregular bipartite graph and let $V(G)=(P, Q)$ with $|P|=p$ and $|Q|=n-p$. Let $H_1 \cong H_2\cong \ldots \cong H_p \cong \overline{K_{n_1}}$ and $H_{p+1}\cong \ldots \cong H_n \cong \overline{K_{n_2}}$. If $|V(\overline{K_{n_1}})|=n_{1}$ and $|V(\overline{K_{n_{2}}})|=n_{2}$, then for any real $\alpha\in[0,1]$,  we see
\begin{align*}
&f_{A_{\alpha}(G\tilde{o}\wedge^{n}_{i=1}H_{i})}(\lambda)
=(\lambda-\alpha)^{pn_{1}+(n-p)n_{2}}\cdot\left(\lambda-\alpha n_{2}-\alpha r_{2}-(1-\alpha)^{2}\frac{n_{2}}{\lambda-\alpha}\right)^{n-2p}\\
 &\cdot(1-\alpha)^{2p}\cdot l\left(\frac{\sqrt{\big(\lambda-\alpha n_{1}-\alpha r_{1}-(1-\alpha)^{2}\frac{n_{1}}{\lambda-\alpha}\big)\cdot\big(\lambda-\alpha n_{2}-\alpha r_{2}-(1-\alpha)^{2}\frac{n_{2}}{\lambda-\alpha}\big)}}{1-\alpha}\right),
\end{align*}
or
\begin{align*}
&f_{A_{\alpha}(G\tilde{o}\wedge^{n}_{i=1}H_{i})}(\lambda)
=(\lambda-\alpha)^{pn_{1}+(n-p)n_{2}}\cdot\left(\lambda-\alpha n_{1}-\alpha r_{1}-(1-\alpha)^{2}\frac{n_{1}}{\lambda-\alpha}\right)^{2p-n}\\
 &\cdot(1-\alpha)^{2(n-p)}\cdot h\left(\frac{\sqrt{\big(\lambda-\alpha n_{1}-\alpha r_{1}-(1-\alpha)^{2}\frac{n_{1}}{\lambda-\alpha}\big)\cdot\big(\lambda-\alpha n_{2}-\alpha r_{2}-(1-\alpha)^{2}\frac{n_{2}}{\lambda-\alpha}\big)}}{1-\alpha}\right).
\end{align*}
\end{corollary}

\begin{proof}
Since $\overline{K_{n_1}}$ and $\overline{K_{n_2}}$ are two graphs with $n_{1}$ isolated vertices and $n_{2}$ isolated vertices, respectively, the vertices of both graphs have zero degrees. Thus, by Lemma \ref{r-regular}, we find that
$$\chi_{A_{\alpha}(\overline{K_{n_1}})}(\lambda)=\frac{n_{1}}{\lambda-\alpha},~~~~~~~~\\
\chi_{A_{\alpha}(\overline{K_{n_2}})}(\lambda)=\frac{n_{2}}{\lambda-\alpha}.$$
Hence, by Theorem \ref{pq bipartite}, we have
\begin{align*}
&f_{A_{\alpha}(G\tilde{o}\wedge^{n}_{i=1}H_{i})}(\lambda)
=(\lambda-\alpha)^{pn_{1}+(n-p)n_{2}}\cdot\left(\lambda-\alpha n_{2}-\alpha r_{2}-(1-\alpha)^{2}\frac{n_{2}}{\lambda-\alpha}\right)^{n-2p}\\
 &\cdot(1-\alpha)^{2p}\cdot l\left(\frac{\sqrt{\big(\lambda-\alpha n_{1}-\alpha r_{1}-(1-\alpha)^{2}\frac{n_{1}}{\lambda-\alpha}\big)\cdot\big(\lambda-\alpha n_{2}-\alpha r_{2}-(1-\alpha)^{2}\frac{n_{2}}{\lambda-\alpha}\big)}}{1-\alpha}\right),
\end{align*}
or
\begin{align*}
&f_{A_{\alpha}(G\tilde{o}\wedge^{n}_{i=1}H_{i})}(\lambda)
=(\lambda-\alpha)^{pn_{1}+(n-p)n_{2}}\cdot\left(\lambda-\alpha n_{1}-\alpha r_{1}-(1-\alpha)^{2}\frac{n_{1}}{\lambda-\alpha}\right)^{2p-n}\\
 &\cdot(1-\alpha)^{2(n-p)}\cdot h\left(\frac{\sqrt{\big(\lambda-\alpha n_{1}-\alpha r_{1}-(1-\alpha)^{2}\frac{n_{1}}{\lambda-\alpha}\big)\cdot\big(\lambda-\alpha n_{2}-\alpha r_{2}-(1-\alpha)^{2}\frac{n_{2}}{\lambda-\alpha}\big)}}{1-\alpha}\right).
\end{align*}
\end{proof}

Let $G$ be a semiregular bipartite graph and $H_1$, $H_2$, \ldots, $H_n$ be the regular or the complete bipartite graphs. Then combining Lemmas \ref{r-regular}, \ref{chi} with the above conclusions, it is easy to obtain the $A_{\alpha}$-characteristic polynomial of $G\tilde{o}\wedge^{n}_{i=1} H_{i}$, for any real $\alpha\in[0,1]$.
\begin{corollary}\label{ni}
Let $G$ be an $(r_1,r_2)$-semiregular bipartite graph and let $V(G)=(P, Q)$ with $|P|=p$ and $|Q|=n-p$. Let $H_{1} \cong H_2\cong \ldots \cong H_{p} \cong Z_{1}$ and $H_{p+1} \cong \ldots \cong H_{n} \simeq Z_{2}$. If $Z_{1}$ is an $r_{1}'$-regular graph with $|V(Z_{1})|=n_{1}$ and $Z_{2}$ is an $r_{2}'$-regular graph with $|V(Z_{2})|=n_{2}$, then for any real $\alpha\in[0,1]$, we see that
\begin{align*}
&f_{A_{\alpha}(G\tilde{o}\wedge^{n}_{i=1}H_{i})}(\lambda)\\
&=\big(f_{A_{\alpha}(Z_{1})}(\lambda-\alpha)\big)^p\cdot\big(f_{A_{\alpha}(Z_{2})}(\lambda-\alpha)\big)^{n-p}
\cdot\left(\lambda-\alpha n_{2}-\alpha r_{2}-(1-\alpha)^{2}\frac{n_{2}}{\lambda-\alpha-r_{2}'}\right)^{n-2p}\\
 &\cdot(1-\alpha)^{2p}\cdot l\left(\frac{\sqrt{\big(\lambda-\alpha n_{1}-\alpha r_{1}-(1-\alpha)^{2}\frac{n_{1}}{\lambda-\alpha-r_{1}'}\big)\cdot\big(\lambda-\alpha n_{2}-\alpha r_{2}-(1-\alpha)^{2}\frac{n_{2}}{\lambda-\alpha-r_{2}'}\big)}}{1-\alpha}\right),
\end{align*}
or
\begin{align*}
&f_{A_{\alpha}(G\tilde{o}\wedge^{n}_{i=1}H_{i})}(\lambda)\\
&=\big(f_{A_{\alpha}(Z_{1})}(\lambda-\alpha)\big)^p\cdot\big(f_{A_{\alpha}(Z_{2})}(\lambda-\alpha)\big)^{n-p}
\cdot\left(\lambda-\alpha n_{1}-\alpha r_{1}-(1-\alpha)^{2}\frac{n_{1}}{\lambda-\alpha-r_{1}'}\right)^{2p-n}\\
 &\cdot(1-\alpha)^{2(n-p)}\cdot h\left(\frac{\sqrt{\big(\lambda-\alpha n_{1}-\alpha r_{1}-(1-\alpha)^{2}\frac{n_{1}}{\lambda-\alpha-r_{1}'}\big)\cdot\big(\lambda-\alpha n_{2}-\alpha r_{2}-(1-\alpha)^{2}\frac{n_{2}}{\lambda-\alpha-r_{2}'}\big)}}{1-\alpha}\right).
\end{align*}
\end{corollary}


\begin{corollary}\label{MS bipartite}
Let $G$ be an $(r_1,r_2)$-semiregular bipartite graph and let $V(G)=(P, Q)$ with $|P|=p$ and $|Q|=n-p$. Let $H_{1} \cong H_2\cong \ldots \cong H_{p} \cong K_{p_{1},q_{1}}$ and $H_{p+1} \cong \ldots \cong H_{n} \cong K_{p_{2},q_{2}}$. Then for any real $\alpha\in[0,1]$, we know that
\begin{align*}
f_{A_{\alpha}(G\tilde{o}\wedge^{n}_{i=1}H_{i})}(\lambda)
=\big(f_{A_{\alpha}(Z_{1})}(\lambda-\alpha)\big)^p\cdot\big(f_{A_{\alpha}(Z_{2})}(\lambda-\alpha)\big)^{n-p}
 \cdot b^{n-2p}\cdot(1-\alpha)^{2p}\cdot l\left(\frac{\sqrt{ab}}{1-\alpha}\right),
\end{align*}
or
\begin{align*}
f_{A_{\alpha}(G\tilde{o}\wedge^{n}_{i=1}H_{i})}(\lambda)
=\big(f_{A_{\alpha}(Z_{1})}(\lambda-\alpha)\big)^p\cdot\big(f_{A_{\alpha}(Z_{2})}(\lambda-\alpha)\big)^{n-p}
 \cdot a^{2p-n}\cdot(1-\alpha)^{2(n-p)}\cdot h\left(\frac{\sqrt{ab}}{1-\alpha}\right),
\end{align*}
where
$$a=\lambda-\alpha(p_{1}+q_{1})-\alpha r_{1} -(1-\alpha)^{2}\frac{(p_{1}+q_{1})\lambda-\alpha (p_{1}+q_{1})^{2}+2p_{1}q_{1}}{\lambda^{2}
 -\alpha (p_{1}+q_{1})\lambda+(2\alpha-1)p_{1}q_{1}},$$

$$b=\lambda-\alpha (p_{2}+q_{2})-\alpha r_{2} -(1-\alpha)^{2}\frac{(p_{2}+q_{2})\lambda-\alpha (p_{2}+q_{2})^{2}+2p_{2}q_{2}}{\lambda^{2}
 -\alpha (p_{2}+q_{2})\lambda+(2\alpha-1)p_{2}q_{2}}.$$
\end{corollary}


\begin{corollary}
Let $G$ be an $(r_1,r_2)$-semiregular bipartite graph and let $V(G)=(P, Q)$ with $|P|=p$ and $|Q|=n-p$. Let $H_{1} \cong H_2\cong \ldots \cong H_{p} \cong Z_{1}$ and $H_{p+1} \cong \ldots \cong H_{n} \cong K_{p_{2},q_{2}}$. If $Z_{1}$ is an $r_{1}'$-regular graph with $|V(Z_{1})|=n_{1}$, then for any real $\alpha\in[0,1]$, we have
\begin{align*}
f_{A_{\alpha}(G\tilde{o}\wedge^{n}_{i=1}H_{i})}(\lambda)
=\big(f_{A_{\alpha}(Z_{1})}(\lambda-\alpha)\big)^p\cdot\big(f_{A_{\alpha}(Z_{2})}(\lambda-\alpha)\big)^{n-p}
\cdot b^{n-2p}\cdot(1-\alpha)^{2p}\cdot l\left(\frac{\sqrt{bc}}{1-\alpha}\right),
\end{align*}
or
\begin{align*}
f_{A_{\alpha}(G\tilde{o}\wedge^{n}_{i=1}H_{i})}(\lambda)
=\big(f_{A_{\alpha}(Z_{1})}(\lambda-\alpha)\big)^p\cdot\big(f_{A_{\alpha}(Z_{2})}(\lambda-\alpha)\big)^{n-p}
\cdot c^{2p-n}\cdot(1-\alpha)^{2(n-p)}\cdot h\left(\frac{\sqrt{bc}}{1-\alpha}\right),
\end{align*}
where
$$b=\lambda-\alpha (p_{2}+q_{2})-\alpha r_{2}-(1-\alpha)^{2}\frac{(p_{2}+q_{2})\lambda-\alpha (p_{2}+q_{2})^{2}+2p_{2}q_{2}}{\lambda^{2}-\alpha (p_{2}+q_{2})\lambda+(2\alpha-1)p_{2}q_{2}},$$

$$c=\lambda-\alpha n_{1}-\alpha r_{1}-(1-\alpha)^{2}\frac{n_{1}}{\lambda-\alpha-r_{1}'}.$$
\end{corollary}

\section{The $A_{\alpha}$-characteristic polynomial of $G\left[H_i\right]_1^m$}
In this section, we will focus on the $A_{\alpha}$-characteristic polynomial of the generalized edge corona graph and give some applications of related results. Before proving our main results, we first give some notation which will be useful later.

Let $G$ be a graph with $n$ vertices and $m$ edges. Then the vertex-edge {\it incidence matrix} $R(G)=(r_{ij})$ of $G$ is an $n\times m$ matrix with entry $r_{ij}=1$ if the vertex $v_{i}$ is incident to the edge $e_{j}$ and $0$ otherwise. Let $A=(a_{ij})$ and $B$ be two matrices of order $m\times n$ and $p\times q$, respectively. The {\it Kronecker product} of $A$ and $B$, denoted $A\otimes B$, is the $mp\times nq$ block matrix $(a_{ij}B)$. It can be verified from definition that
$$(A\otimes B)(C\otimes D)=AC\otimes BD.$$

Let $G$ be an $r_{1}$-regular graph with $n$ vertices and $m$ edges, and $H_{1}, H_{2}, ...,H_{m}$ be $m$ $r_{2}$-regular graphs with $q$ vertices. By Definition \ref{GEC}, the vertices of $G[H_i]_1^m$ are partitioned by $V(G)\cup V(H_1)\cup V(H_2)\cup \cdots \cup V(H_m)$. Then the adjacency matrix and degree matrix of $G[H_i]_1^m$ are given as follows, respectively.
\begin{equation}\label{bianA}
A(G[H_i]_1^m):=\begin{pmatrix}
A(G) & R\otimes \mathbf{1}_{q}^{T} \\
 R^{T}\otimes \mathbf{1}_{q} & F
\end{pmatrix},
\end{equation}
where
\[ F=\begin{array}{c}
V(H_1) \\
V(H_2) \\
\vdots \\
V(H_{m-1}) \\
V(H_m)
\end{array} \bordermatrix{
       &V(H_{1})&V(H_{2})&\cdots&V(H_{m-1})&V(H_{m})\cr
 &A(H_{1}) & 0 & 0 & \cdots & 0 \cr
 &0 & A(H_{2}) & 0 & \cdots & 0 \cr
 &\vdots & \vdots & \ddots & \vdots & \vdots \cr
&0 & 0 & \cdots & A(H_{m-1}) & 0 \cr
 &0 & 0 & \cdots & 0 & A(H_{m}) \cr} . \]

\begin{small}\begin{equation}\label{bianD}
D(G[H_i]_1^m)=\begin{array}{c}
V(G)\\
V(H_1) \\
V(H_2) \\
\vdots \\
V(H_m)
\end{array} \bordermatrix{
  &V(G) &V(H_{1})&V(H_{2})&\cdots&V(H_{m})\cr
  &D(A(G))+r_{1}qI_{n}&0&0&\cdots&0\cr
  &0          &D(H_{1})+2I_{q} & 0 & \cdots & 0 \cr
  &0          &0 &D(H_{2})+2I_{q}  & \cdots& 0  \cr
  &\vdots     &\vdots & \vdots & \ddots & \vdots  \cr
  &0          &0 & 0 & \cdots &D(H_{m})+2I_{q}  \cr }.
\end{equation}
\end{small}

Using (\ref{A-alpha-matrix}), we know
\begin{equation}\label{bianAlpha}
A_{\alpha}(G[H_i]_1^m)=\alpha D(G[H_i]_1^m)+(1-\alpha)A(G[H_i]_1^m).
\end{equation}

Substituting (\ref{bianA}) and (\ref{bianD}) into (\ref{bianAlpha}), we obtain the $A_{\alpha}$-matrix of $G[H_i]_1^m$ below
\begin{equation}\label{bianAlpha-2}
\begin{aligned}
A_{\alpha}(G[H_i]_1^m)=\begin{pmatrix}
\alpha r_1 q I_n+A_{\alpha}(G) & (1-\alpha)R \otimes \mathbf{1}_{q}^T \\
(1-\alpha)R^T \otimes \mathbf{1}_{q} & 2\alpha I_{qm}+A_{\alpha}(H_{i}|_{1}^{m})
\end{pmatrix},
\end{aligned}
\end{equation}
where
$$A_{\alpha}(H_{i}|_{1}^{m})=
\left(
\begin{array}{ccccccccccccc}
& A_{\alpha}(H_1) & 0 &\cdots & 0 \\
& 0 & A_{\alpha}(H_2)& \cdots & 0 \\
& 0 & 0 & \cdots & 0 \\
&0 & 0 & \cdots &A_{\alpha}(H_m)
\end{array}
\right).
$$

Therefore, assume that $G$ is an $r_{1}$-regular graph with $n$ vertices and $m$ edges, and $H_{1}, H_{2}, ...,H_{m}$ are $r_{2}$-regular graphs with $q$ vertices. According to the notations above, we set
$$W(\lambda)=R\otimes \mathbf{1}_{q}^{T}\big((\lambda-2\alpha)I_{qm}-A_{\alpha}(H_{i}|_{1}^{m})\big)^{-1}R^{T}\otimes \mathbf{1}_{q}.$$

\begin{proposition}\label{bianguan}
Let $G$ be an $r_1$-regular graph with $n$ vertices and $m$ edges, $H_1, H_2, \ldots, H_m$ be $r_2$-regular graphs, each of them has $q$ vertices. Then
$$W(\lambda)=\frac{q}{\lambda-2\alpha-r_2}\big(A(G)+r_1 I_n\big).$$
\end{proposition}

\begin{proof}
Since each row sum of $A_{\alpha}(H_{i}|_{1}^{m})$ equals $r_2$,  $A_{\alpha}(H_{i}|_{1}^{m})(R^T \otimes \mathbf{1}_{q})=r_2(R^T \otimes \mathbf{1}_{q})$. Hence
$$
\big((\lambda- 2\alpha) I_{q m}-A_{\alpha}(H_{i}|_{1}^{m})\big)(R^T \otimes \mathbf{1}_{q})=(\lambda-2\alpha-r_2)(R^T \otimes \mathbf{1}_{q}).
$$
Therefore,
\begin{equation*}
(R\otimes \mathbf{1}_{q}^{T})\big((\lambda-2\alpha)I_{qm}-A_{\alpha}(H_{i}|_{1}^{m})\big)^{-1}(R^{T}\otimes \mathbf{1}_{q})=\frac{(R \otimes \mathbf{1}_{q}^T) \cdot(R^T \otimes \mathbf{1}_{q})}{\lambda-2\alpha-r_2}=\frac{q}{\lambda-2\alpha-r_2}\big(A(G)+r_1 I_n\big),
\end{equation*}
as required.
\end{proof}

We are now in a position to prove the next result, which gives the expression of the $A_{\alpha}$-characteristic polynomial of $G[H_i]_1^m$.

\begin{theorem}\label{biantheorem}
Let $G$ be an $r_{1}$-regular graph with $n$ vertices and $m$ edges, $H_1$, $H_2, \ldots, H_m$ be $r_2$-regular graphs, each of them has $q$ vertices. Then for any real $\alpha\in[0,1]$, we get the $A_{\alpha}$-characteristic polynomial of $G[H_i]_1^m$ as follows:
$$
\begin{aligned}
f_{A_{\alpha}(G[H_i]_1^m)}(\lambda)
=&\frac{1}{(\lambda-2\alpha-r_{2})^n} \prod_{i=1}^m f_{A_{\alpha}(H_i)}(\lambda-2\alpha)\cdot\prod_{j=1}^n\big(\lambda^{2}-\big(2\alpha+r_{2}+\alpha r_{1}(q+1)+(1-\alpha)\lambda_{j}\big)\lambda\\
&+(2\alpha+r_{2})\big(\alpha r_{1}(q+1)+(1-\alpha)\lambda_{j}\big)-q(1-\alpha)^{2}(r_{1}+\lambda_{j})\big),
\end{aligned}
$$
where $\lambda_{1} \leq \lambda_{2} \leq \cdots \leq \lambda_{n}=r_{1}$ are the $A$-eigenvalues of $G$.
\end{theorem}

\begin{proof}
According to the $A_{\alpha}$-matrix of $G[H_i]_1^m$ from (\ref{bianAlpha-2}), it follows that
$$
\begin{aligned}
f_{A_{\alpha}(G[H_i]_1^m)}(\lambda)&=\operatorname{det}\big(\lambda I-A_{\alpha}(G[H_i]_1^m)\big)\\
&=\operatorname{det}\left(\begin{array}{cc}
\lambda I_n-\alpha r_1 q I_n-A_{\alpha}(G) & -(1-\alpha)R \otimes \mathbf{1}_{q}^T \\
-(1-\alpha)R^T \otimes \mathbf{1}_{q} & \lambda I_{q m}-2\alpha I_{qm}-A_{\alpha}(H_{i}|_{1}^{m})
\end{array}\right) .
\end{aligned}
$$
By Lemma \ref{Schur}, one may obtain that
$$
\begin{aligned}
&f_{A_{\alpha}(G[H_i]_{1}^{m})}(\lambda)  \\
= & \operatorname{det}\left(\begin{array}{llll}
(\lambda-2\alpha) I_{q}-A_{\alpha}(H_1) & & & \\
  & (\lambda-2\alpha) I_{q}-A_{\alpha}(H_2) & & \\
  & & \ddots&  \\
  & & & (\lambda-2\alpha) I_{q}-A_{\alpha}(H_m)
\end{array}\right) \\
& \cdot \operatorname{det}\left(\lambda I_{n}-\alpha r_{1} q  I_{n}-A_{\alpha}(G)-(1-\alpha)^{2}(R\otimes \mathbf{1}_{q}^{T})\big((\lambda-2\alpha)I_{qm}-A_{\alpha}(H_{i}|_{1}^{m})\big)^{-1}(R^{T}\otimes \mathbf{1}_{q})\right)\\
=& \prod_{i=1}^m f_{A_{\alpha}(H_i)}(\lambda-2\alpha) \\
& \cdot \operatorname{det}\left(\lambda I_{n}-\alpha r_{1} q  I_{n}-A_{\alpha}(G)-(1-\alpha)^{2}(R\otimes \mathbf{1}_{q}^{T})\big((\lambda-2\alpha)I_{qm}-A_{\alpha}(H_{i}|_{1}^{m})\big)^{-1}(R^{T}\otimes \mathbf{1}_{q})\right).
\end{aligned}
$$

Therefore, as $G$ is an $r_1$ regular graph and by Proposition \ref{bianguan}, we obtain that
$$
\begin{small}
\begin{aligned}
&f_{A_{\alpha}(G[H_i]_{1}^{m})}(\lambda)  \\
=& \prod_{i=1}^m f_{A_{\alpha}(H_i)}(\lambda-2\alpha) \cdot \operatorname{det}\left(\lambda I_{n}-\alpha r_{1}qI_{n}-A_{\alpha}(G)-(1-\alpha)^{2}\frac{q}{\lambda-2 \alpha-r_{2}}\big(A(G)+r_{1}I_{n}\big)\right)\\
=& \prod_{i=1}^m f_{A_{\alpha}(H_i)}(\lambda-2\alpha) \cdot \operatorname{det}\left(\lambda I_{n}-\alpha r_{1}qI_{n}-\alpha r_{1}I_{n}-(1-\alpha)A(G)-\frac{(1-\alpha)^{2}q}{\lambda-2 \alpha-r_{2}}\big(A(G)+r_{1}I_{n}\big)\right)\\
=&\prod_{i=1}^m f_{A_{\alpha}(H_i)}(\lambda-2\alpha) \cdot \operatorname{det}\left(\lambda I_{n}-\alpha r_{1}(q+1)I_{n}-\frac{qr_{1}(1-\alpha)^{2}}{\lambda-2\alpha-r_{2}}I_{n}-
\frac{(1-\alpha)\big(\lambda-2\alpha-r_{2}+(1-\alpha)q\big)}{\lambda-2\alpha-r_{2}}A(G)\right)\\
= & \prod_{i=1}^m f_{A_{\alpha}(H_i)}(\lambda-2\alpha) \cdot \prod_{j=1}^n\left(\lambda-\alpha r_{1}(q+1)-\frac{qr_{1}(1-\alpha)^{2}}{\lambda-2\alpha-r_{2}}-
\frac{\lambda_{j}(1-\alpha)\big(\lambda-2\alpha-r_{2}+(1-\alpha)q\big)}{\lambda-2\alpha-r_{2}}\right) \\
= & \frac{1}{(\lambda-2\alpha-r_{2})^n} \prod_{i=1}^m f_{A_{\alpha}(H_i)}(\lambda-2\alpha)\cdot\prod_{j=1}^n\big(\lambda^{2}-\big(2\alpha+r_{2}+\alpha r_{1}(q+1)+(1-\alpha)\lambda_{j}\big)\lambda\\
&+(2\alpha+r_{2})\big(\alpha r_{1}(q+1)+(1-\alpha)\lambda_{j}\big)-q(1-\alpha)^{2}(r_{1}+\lambda_{j})\big).
\end{aligned}
\end{small}
$$
This completes the proof.
\end{proof}

Assume that the $A$-eigenvalues of $G$ are $\lambda_1 \leq \lambda_2 \leq$ $\cdots \leq \lambda_n=r_1$. For $j=1,2, \ldots, n$, set $b=2\alpha+r_{2}+\alpha r_{1}(q+1)+(1-\alpha)\lambda_{j}$ and
$$
\gamma_j, \bar{\gamma}_j=\frac{b \pm \sqrt{b^2-4\big((2\alpha+r_{2})\big(\alpha r_{1}(q+1)+(1-\alpha)\lambda_{j}\big)-q(1-\alpha)^{2}(r_{1}+\lambda_{j})\big)}}{2}.
$$

Then the following corollary is immediate from the Theorem \ref{biantheorem}.

\begin{corollary} Let $G$ be an $r_1$-regular connected graph with $n$ vertices and $m$ edges, and $H_1, H_2, \ldots, H_m$ be $m$ $r_2$-regular graphs with $q$ vertices. Suppose that the $A_{\alpha}$-eigenvalues of $H_i$ are $\mu_1^{(i)} \leq \mu_2^{(i)} \leq \cdots \leq$ $\mu_{q}^{(i)}=r_2$ for $i=1,2, \ldots, m$. Then $A_{\alpha}$-spectrum of $G[H_i]_1^m$ is
$$\left(
\begin{array}{ccccccccccccc}
2\alpha+r_{2} & \mu_{1}^{(1)}+2\alpha & \cdots & \mu_{q-1}^{(1)}+2\alpha & \cdots &  \mu_1^{(m)}+2\alpha & \cdots & \mu_{q-1}^{(m)}+2\alpha & \gamma_1 & \bar{\gamma}_1 & \cdots & \gamma_n & \bar{\gamma}_n\\
m-n & 1 & \cdots & 1 & \cdots & 1 & \cdots & 1 & 1 & 1 & \cdots & 1 & 1
\end{array} \right),
$$
where entries in the first row are the eigenvalues with the number of repetitions written below, respectively.
\end{corollary}

\begin{corollary}
Let $G_{1}$ and $G_{2}$ be two $A$-cospectral $r_{1}$-regular graphs
with $n$ vertices and $m$ edges, and $H_{1}, \ldots, H_{m}$ be $r_{2}$-regular graphs and each of them have the same order $q$. Then for any real $\alpha\in[0,1]$, $G_{1}[H_i]_{1}^{m}$ and $G_{2}[H_i]_{1}^{m}$ are $A_{\alpha}$-cospectral.
\end{corollary}

\begin{corollary}
Let $G$ be an $r_{1}$-regular graph with $n$ vertices and $m$ edges, and $H_{1}, \ldots, H_{m}$, $F_{1}, \ldots, F_{m}$ be $2m$ $r_{2}$-regular graphs each of them with same order $q$. Suppose that $H_{i}$ and $F_{i}$ are $m$ pairs of $A_{\alpha}$-cospectral graphs for $i=1,2, \ldots, m$. Then for any real $\alpha\in[0,1]$, $G[H_i]_1^m$ and $G[F_{i}]_{1}^{m}$ are $A_{\alpha}$-cospectral graphs.
\end{corollary}

\end{document}